\documentclass[11pt, a4paper]{amsart}

\usepackage{amsmath,amsthm,amssymb}
\usepackage{txfonts, mathabx}
\usepackage[colorlinks,urlcolor=blue,linkcolor=red,citecolor=blue]{hyperref}

\newtheorem{theorem}{Theorem}[section]
\newtheorem{lemma}[theorem]{Lemma}

\newtheorem{prop}[theorem]{Proposition}
\newtheorem{coro}[theorem]{Corollary}

\theoremstyle{definition}
\newtheorem{rem}[theorem]{Remark}
\newtheorem{rems}[theorem]{Remarks}

\newtheorem{exas}[theorem]{Examples}

\newcommand\HH{\mathcal H}

\newcommand\tr{\mathrm{tr}}
\newcommand\Id{\mathrm{Id}}
\newcommand\ZZ{\mathbb{Z}}

\newcommand\CC{\mathbb{C}}
\newcommand\FF{\mathbb{F}}
\newcommand\Pg{\overline{P}}
\newcommand\TT{\overline{T}}
\newcommand\IS{\mathrm{IS}}

\numberwithin{equation}{section}

\title[Intertwined integer sequences]
{Pairs of intertwined integer sequences}

\author{Christian Kassel}
\address{Christian Kassel: 
Universit\'e de Strasbourg \& CNRS, 
Institut de Recherche Math\'ematique Avanc\'ee (IRMA),
7~rue Ren\'e-Descartes,
67084 Strasbourg, France}
\email{kassel@math.unistra.fr}
\urladdr{irma.math.unistra.fr/\raise-2pt\hbox{\~{}}kassel/}

\author{Christophe Reutenauer}
\address{Christophe Reutenauer:
Math\'ematiques, Universit\'e du Qu\'ebec \`a Montr\'eal,
Montr\'eal, CP~8888, succ.\ Centre Ville, Canada H3C 3P8}
\email{reutenauer.christophe@uqam.ca}
\urladdr{reutenauer.math.uqam.ca}



\begin{document}

\begin{abstract}
In previous work we computed the number~$C_n(q)$ of ideals of codimension~$n$ of 
the algebra $\FF_q[x,y,x^{-1}, y^{-1}]$ of two-variable Laurent polynomials over a finite field: it turned out that 
$C_n(q)$ is a palindromic polynomial of degree~$2n$ in~$q$, divisible by~$(q-1)^2$. The quotient $P_n(q) = C_n(q)/ (q-1)^2$
is a palindromic polynomial of degree~$2n-2$. For each $n\geq 1$ there is a unique degree~$n-1$ polynomial~$\Pg_n(X) \in \ZZ[X]$
such that $\Pg_n(q+q^{-1}) = P_n(q)/q^{n-1}$. 
In this note we show that for any integer~$N$ the integer value $\Pg_n(N)$ is close to the value at~$N$ of the 
degree~$n-1$ polynomial $F_{n-1}(X) = 1 + \sum_{k=1}^{n-1} \, \TT_k(X)$,
which is a sum of monic versions $\TT_k(X)$ of Chebyshev polynomials of the first kind.
We give a precise formula for~$\Pg_n(X)$ as a linear combination of $F_k(X)$'s, each appearance of the latter being
parametrized by an odd divisor of~$n$. As a consequence, $\Pg_n(X) = F_{n-1}(X)$ if and only if $n$ is a power of~$2$.
We exhibit similar formulas for~$C_n(q)$.
\end{abstract}

\maketitle

%%%
\section{Introduction}\label{sec-intro}

In this note we produce for each integer~$N$ a pair $(\Pg_n(N), (F_{n-1}(N))_{n\geq 1}$ of integer sequences.
Both sequences, though of differing origin,
consist of amazingly close elements. Some of them do even coincide: we have $\Pg_n(N) = F_{n-1}(N)$ whenever $n$ is a power of~$2$.
The reader is invited to examine Table~\ref{table-values} (manufactured using SageMath) which exhibits such pairs for $N = 3,4,5$; 
in this table coinciding values have been set in boldface and values differing by~$1$ in italics.

\begin{table}[ht]
\caption{\emph{Values of $\Pg_n(N)$ and $F_{n-1}(N)$ for $N = 3,4,5$}}\label{table-values}
\renewcommand\arraystretch{1.30}
\noindent\[\hskip -14pt
\begin{array}{|c||c|c||c|c||c|c||}
\hline
n &  \Pg_n(3) & F_{n-1}(3) &  \Pg_n(4) & F_{n-1}(4) & \Pg_n(5) & F_{n-1}(5) \\
\hline\hline
{1}  & \textbf{1} &  \textbf{1} &\textbf{1}  & \textbf{1}  & \textbf{1} & \textbf{1}\\ 
\hline
{2} &  \textbf{4} & \textbf{4} & \textbf{5} & \textbf{5} & \textbf{6}  & \textbf{6} \\
\hline
3 &  \textit{10} & \textit{11} &  \textit{18} & \textit{19} & \textit{28} & \textit{29} \\ 
\hline
{4} &  \textbf{29}  & \textbf{29}  & \textbf{71} & \textbf{71} & \textbf{139} & \textbf{139}  \\ 
\hline
5 &   72 & 76 & 260 & 265 & 660 & 666\\ 
\hline
6 &    \textit{200} &  \textit{199} & \textit{990} & \textit{989} & \textit{3192} & \textit{3191}  \\
\hline
7 &  510  & 521  &  3672 & 3691 & 15260 & 15289 \\ 
\hline
{8} &  \textbf{1364} & \textbf{1364} & \textbf{13775} & \textbf{13775} & \textbf{73254} & \textbf{73254} \\ 
\hline
9 &  3546  & 3571  & 51343 & 51409 & 350848 &  350981 \\
\hline
10 &  \textit{9348}  &  \textit{9349}  & \textit{191860} & \textit{191861} & \textit{1681650} & \textit{1681651}  \\
\hline
11 &  24400  & 24476  & 715770 & 716035 & 8056608 & 8057274 \\
\hline
12 &  64090 & 64079 & 2672298 & 2672279 & 38604748 & 38604719 \\
\hline
13 & 167562  & 167761 & 9972092 & 9973081 & 184963130 & 184966321 \\
\hline
14 & 439200  & 439204 & 37220040 & 37220045 &886226880 & 886226886 \\
\hline
15 &  1149360 &1149851  & 138903480 & 138907099 &4246152960 & 4246168109 \\
\hline
{16} & \textbf{3010349} & \textbf{3010349} & \textbf{518408351} & \textbf{518408351} & \textbf{20344613659} & \textbf{20344613659}\\
\hline
\end{array}
\]
\end{table}

We now describe the components of these pairs. Starting with~$\Pg_n(N)$,
let~$\FF_q$ be a finite field of cardinality~$q$ and $\FF_q[x,y,x^{-1}, y^{-1}]$ be the algebra of
two-variable Laurent polynomials with coefficients in~$\FF_q$.
For any $n\geq 1$ define $C_n(q)$ to be the number of ideals of codimension~$n$ of~$\FF_q[x,y,x^{-1}, y^{-1}]$; 
such ideals are in one-to-one correspondence with the $\FF_q$-points of the Hilbert scheme 
of $n$~points on the two-dimensional torus.
In~\cite{KR2} we showed that $C_n(q)$ is a monic polynomial of degree~$2n$ with integer coefficients in the variable~$q$.
Moreover, $C_n(q)$ is palindromic and divisible by~$(q-1)^2$. 
We gave explicit formulas for the coefficients of~$C_n(q)$ in~\cite[Th.~1.1]{KR3} (see also \S~\ref{ssec-Cnq} below).

Since $C_n(q)$ is divisible by~$(q-1)^2$ we can define the degree $2n-2$ polynomial~$P_n(q)$ by $C_n(q) = (q-1)^2 P_n(q)$: 
it is monic, palindromic and has integer coefficients. 
One of the most striking features of~$P_n(q)$ is that its coefficients are all \emph{nonnegative}.
Actually we showed that each coefficient of~$P_n(q)$ can be expressed
as the number of divisors of~$n$ in a certain interval (see~\cite[Thm.~1.3]{KR3} and \S~\ref{ssec-defPn} below).

In \cite{KR1, KR3} we also computed the values taken by $P_n(q)$ when $q$ is a complex root of unity of order $1,2,3,4$ or~$6$, 
equivalently when $q + q^{-1}$ takes one of the respective values $2, -2, -1, 0, 1$ (see \cite[Thm.~1.6]{KR3}),
thus recovering well-known arithmetical functions. For instance, $P_n(1) = \sigma(n)$ is equal to the sum of divisors of~$n$
and $P_n(-1) = r(n)/4$, where $r(n)$ is the number of representations of~$n$ as a sum of squares of two integers.

Actually the Laurent polynomial ${P_n(q)}/{q^{n-1}}$ takes integer values whenever $q + q^{-1} = N$ is an integer.
Computing such integer values is our main objective in this note. 
For simplicity we replace the Laurent polynomial $P_n(q)/q^{n-1}$ in the variable~$q$ 
by the degree $n-1$ polynomial~$\Pg_n(X)$ 
in the variable~$X$, defined by
\begin{equation}\label{Pg-P}
\Pg_n(q + q^{-1}) = P_n(q)/q^{n-1} .
\end{equation}
(This is possible since $P_n(q)$ is palindromic of degree~$2n-2$.)
We have thus reduced our problem to the computation of the values of~$\Pg_n(X)$ at integers~$N$.

Plugging the integers $\Pg_1(N), \Pg_2(N),\ldots, \Pg_8(N)$ for $N=5$ into the 
\emph{On-Line Encyclopedia of Integer Sequences}~\cite{OEIS}
yielded  (as of Jan.~7, 2025)  
an ``approximate match'' by the first coefficients of the power series expansion of the rational function
\[
\frac{1+t}{1 - 5t + t^2}.
\]
Replacing $5$ by~$3$ and by~$4$ in this rational function produced similar approximate matches 
for $\Pg_n(3)$ and $\Pg_n(4)$ respectively.
It was then natural to replace $3,4,5$ by an indeterminate~$X$, which we do next.

The second integer sequence mentioned above consists of integer values at~$N$ of the polynomials~$F_k(X)$ 
defined by
\begin{equation}\label{gf-F}
\sum_{k\geq 0} \, F_k(X) \, t^k 
= \frac{1+t}{1 - Xt + t^2} \, .
\end{equation}
It turns out that each polynomial $F_k(X)$ is of degree~$k-1$, has integer coefficients and we have
\begin{equation}\label{eq-FFTT}
F_k(X) = 1 + \sum_{m=1}^{k} \, \TT_m(X) ,
\end{equation}
which is the sum of monic versions~$\TT_m(X)$ of the Chebyshev polynomials of the first kind
defined in \S~\ref{ssec-Tchb} below.

With this notation we establish that each~$\Pg_n(X)$ can be approximated by~$F_{n-1}(X)$ in the sense that 
the difference between these two monic  polynomials, both of degree~$n-1$,
is a polynomial of much lower degree (see Theorem~\ref{th-approx} for a precise statement). 

More interestingly, $\Pg_n(X)$ can be expressed as the following linear combination of~$F_k(X)$,
whose terms are parameterized by the odd divisors $d$ of~$n$:
\begin{equation}\label{formula-PFodd}
\Pg_n(X) = \sum_{d | n, \, d \, \mathrm{odd} \atop r_n(d) \geq 0} \, F_{r_n(d)}(X) 
- \sum_{d | n, \, d \, \mathrm{odd} \atop r_n(d) <0} \, F_{-r_n(d) -1}(X) ,
\end{equation}
where $r_n(d) = {n/d} - {(d+1)/2}$; see Theorem~\ref{th-PFodd}.
For $d=1$ we have $r_n(1) = n-1$, which implies the presence of~$F_{n-1}(X)$ in~\eqref{formula-PFodd} for all $n\geq 1$.

Since the powers of~$2$ are the only integers with a unique odd divisor, namely $d=1$, 
we have $\Pg_n(X) = F_{n-1}(X)$ if and only if $n$ is a power of~$2$.
This explains the coincidences observed in Table~~\ref{table-values}.

We deduce Formula~\eqref{formula-PFodd} from the following new expression for~$C_n(q)$:
\begin{equation}\label{formula-Cnodd}
C_n(q) = q^n \, \sum_{d | n , \, d \, \mathrm{odd}} \, 
\left( q^{\frac{n}{d} - \frac{d-1}{2}} +  q^{-\frac{n}{d} + \frac{d-1}{2}}  
- q^{\frac{n}{d} - \frac{d+1}{2}}  - q^{-\frac{n}{d} + \frac{d+1}{2}} \right) .
\end{equation}
See Theorem~\ref{th-Cn}.

Formula~\eqref{formula-Cnodd} allows us to express the local zeta function and the Hasse--Weil zeta function
of the Hilbert scheme of $n$~points on the two-dimensional torus
as products parameterized by the odd divisors of~$n$ (see \S~\ref{ssec-zeta}).

Let us detail the content of each section.
In Section~\ref{sec-Pn}, for each $n\geq 1$ we define the degree~$n-1$ polynomial~$\Pg_n(X)$ by~\eqref{Pg-P}.
As mentioned above, the integers of interest to us are the values~$\Pg_n(N)$ at arbitrary integers~$N$.
Building on \cite{KR3}, we show that the function $|\Pg_n(N)|$ is multiplicative if and only if $N= -2, -1, 0$ or~$2$.

Section~\ref{sec-aux} is devoted to the monic versions $\TT_k(X)$ of the Chebyshev polynomials of the first kind
and to the polynomials~$F_k(X)$ defined above.

In Section~\ref{sec-resultsP} we state the results on~$\Pg_n(X)$ mentioned above.
Theorem~\ref{th-approx} and Theorem~\ref{th-PFodd} with Formula~\eqref{formula-PFodd} are the main results.
As special cases of~\eqref{formula-PFodd} we give formulas for~$P_n(X)$ in terms of the polynomials~$F_k(X)$ 
for various families of integers~$n$.

Section~\ref{sec-Cnq} is devoted to the polynomials~$C_n(q)$:
we prove~\eqref{formula-Cnodd} and state the formulas for the zeta functions mentioned above.
In Section~\ref{pf-th-PFodd} we use the results of Section~\ref{sec-Cnq} to prove Theorem~\ref{th-PFodd}.

%%%
\section{The polynomials $\Pg_n(X)$}\label{sec-Pn}

\subsection{Definition and basic properties}\label{ssec-defPn}

For each $n \geq 1$ we define the degree $n-1$ polynomial $\Pg_n(X)$ by
\begin{equation}\label{def-Pgn}
\Pg_n(q + q^{-1}) = P_n(q)/q^{n-1} ,
\end{equation}
where $P_n(q)$ was introduced in Section~\ref{sec-intro}. 
This makes sense since $P_n(q)$ is a palindromic polynomial of degree~$2n-2$. 
Since $P_n(q)$ has integer coefficients, so has~$\Pg_n(X)$.

As a consequence of~\cite[Cor.~1.4]{KR2}, we have the equality of formal power series
\begin{equation*}\label{gf-Pnq}
1 + (q + q^{-1} - 2) \sum_{n\geq 1} \, \frac{P_n(q)}{q^{n-1}}  \, t^n
= \prod_{i\geq 1}\, \frac{(1-t^i)^2}{1-(q+q^{-1})t^i + t^{2i}} \, .
\end{equation*}
This translates into the following equality for the polynomials~$\Pg_n (X)$:
\begin{equation}\label{gf-Pgn}
1 + (X - 2) \sum_{n\geq 1} \, \Pg_n (X)  \, t^n
= \prod_{i\geq 1}\, \frac{(1-t^i)^2}{1- X t^i + t^{2i}} \, .
\end{equation}

By \cite[Thm.~1.3]{KR3} each polynomial~$P_n(q)$ is of the form
\begin{equation*}\label{Pn-coeff}
P_n(q) = a_{n,0}\, q^{n-1}  + \sum_{i=1}^{n-1} \, a_{n,i} \, \left( q^{n+i-1} + q^{n-i-1} \right) , 
\end{equation*}
where the coefficient~$a_{n,i}$ is equal to the number of divisors~$d$ of~$n$ satisfying the inequalities
\begin{equation}\label{Pn-coeff}
\frac{i+ \sqrt{2n+i^2}}{{2}} < d \leq i+ \sqrt{2n+i^2} . \qquad (0 \leq i \leq n-1)
\end{equation}
It follows that for each positive integer~$n$,
\begin{equation}\label{PPTT}
\Pg_n(X) = a_{n,0}  + \sum_{i=1}^{n-1} \, a_{n,i} \, \TT_i(X) ,
\end{equation}
where $\TT_i(X)$ is the monic Chebyshev polynomial of degree~$i$ defined by~\eqref{defTT} below.
Since the coefficients~$a_{n,i}$ are nonnegative, each~$\Pg_n(X)$ is a \emph{sum} of monic Chebyshev polynomials~$\TT_k(X)$.

A list of polynomials $\Pg_n(X)$ for $1 \leq n \leq 12$ is given in Table~\ref{tablePP}.
It is based on the corresponding list of $P_n(q)$ in Table~2 of either~\cite{KR2} or~\cite{KR3}.

\begin{table}[ht]
\caption{\emph{The polynomials $\Pg_n (X)$}}\label{tablePP}
\renewcommand\arraystretch{1.25}
\noindent\[
\begin{array}{|c||c|}
\hline
n & \Pg_n(X) \\
\hline\hline
1 & 1   \\ 
\hline
2 & X + 1  \\
\hline
3 & X^2 + X - 2  \\ 
\hline
4 & X^3 + X^2 - 2X -1  \\ 
\hline
5 & X^4 + X^3 - 3 X^2 - 3X  \\ 
\hline
6 &  X^5 + X^4 - 4X^3 - 3 X^2 + 3X + 2   \\
\hline
7 & X^6 + X^5 - 5 X^4 -4 X^3 + 5X^2 + 2X  \\ 
\hline
8 &  X^7 +X^6 - 6X^5 - 5 X^4 + 10 X^3 + 6 X^2 - 4X -1  \\ 
\hline
9 &  X^8 + X^7 - 7 X^6 - 6 X^5 +15 X^4 + 9 X^3 - 11 X^2 - X +3  \\
\hline
10 & X^9 + X^8 -8 X^7 -7 X^6 + 21 X^5  + 15 X^4 - 20 X^3 - 10 X^2 + 5X   \\
\hline
&  X^{10} + X^9 - 9 X^8 - 8 X^7 + 28 X^6 + 21 X^5  \\
11 &   - 36 X^4 - 21 X^3 + 18 X^2 + 7X - 2  \\
\hline
&   X^{11} + X^{10} - 10 X^9 - 9 X^8 + 36 X^7 + 28 X^6  \\
12 & - 56 X^5 -35 X^4 + 35 X^3 + 16 X^2 - 5 X  -2  \\
\hline
\end{array}
\]
\end{table}

The following result explains why all entries in Table~\ref{table-values} are positive.

\begin{prop}
For each $n \geq 1$, we have $\Pg_n(N) > 0$ if $N$ is an integer $\geq 2$.
\end{prop}

\begin{proof}
It is a consequence of Proposition~\ref{TTpos} below and 
of the fact mentioned above that each $\Pg_n(X)$ is a sum of~$\TT_k(X)$'s.
\end{proof}

\subsection{Multiplicative values of $|P_n(q)|$ at roots of unity}\label{ssec-mult}

The content of this subsection is an addendum to \cite[Th.\,1.6]{KR3}. It concerns the cases when $q + q^{-1} = N$
is an integer such that $-2 \leq N \leq 2$, 
equivalently when $q$ is a root of unity of order $d = 1, 2, 3, 4$ or~$6$. 

\begin{prop}\label{prop-mult}
Let $q$ be a complex root of unity of order $d = 1, 2, 3$ or $4$. 
Then the sequence of nonnegative integers 
\begin{equation*}
\left |P_n(q) \right |  = \left | P_n(q)/q^{n-1} \right | =  \left |\Pg_n(q + q^{-1}) \right |
\end{equation*}
is multiplicative as a function of~$n$.
\end{prop}

Recall that a function $f(n)$ defined on nonnegative integers is \emph{multiplicative}  if $f(mn) = f(m) f(n)$ 
whenever $m$ and $n$ are coprime.

\begin{proof}
(i) When $q = 1$, by \cite[Cor.\, 1.2]{KR2} we have $\Pg_n(2) = P_n(1) = \sigma(n)$, the sum of divisors of~$n$,
which is well-known to be multiplicative. See \cite[Sect.\,2.13]{Ap} and~\cite[A000203]{OEIS}.

(ii) For $q = -1$ we have $\Pg_n(-2) = P_n(-1) = r(n)/4$, where $r(n)$ is the number of representations of~$n$ 
as a sum of squares of two integers. By \cite[A002654]{OEIS} the sequence $r(n)/4$ is multiplicative. 
Then so is~$P_n(-1)$.

(iii) For $q = j$, a root of unity of order~$3$, we have
$\Pg_n(-1) = P_n(j)/j^{n-1} = \lambda(n)$, where $\lambda(n)$ is the sequence A113063 in~\cite{OEIS}.
By \cite[Sect.\,32, p.\,79]{Fi} $\lambda(n)$ is multiplicative. See also \cite{RC2}.

(iv) Let $i$ be a square root of~$-1$. It follows again from~\cite[Th.\,1.6]{KR3} that we have
\[
|\Pg_n(0)| = |P_n(i)/i^{n-1}| = r'(n)/2,
\] 
where $r'(n)$ is the number of representations of~$n$ as a sum of a
square and twice another square. The sequence $r'(n)/2$ is the sequence A002325 in~\cite{OEIS}.
It is known to be multiplicative.
\end{proof}

The case when $q = \omega$ is a root of unity of order~$6$ is slightly more involved. 

\begin{prop}\label{prop-mult1}
Let $\omega$ be a root of unity of order~$6$. 
Then the sequence of nonnegative integers $|P_n(\omega)| =  |\Pg_n(1)|$ is almost multiplicative 
in the sense that for all pairs $(m,n)$ of coprime positive integers,
\begin{equation*}
|P_m(\omega)| |P_n(\omega)| = |P_{mn}(\omega)| \quad \text{when} \;\; (m,n) \equiv (0,1), (1,1)  \;  \text{or} \; (1,2) \pmod{3},
\end{equation*}
\begin{equation*}
|P_m(\omega)| |P_n(\omega)| = 2 \, |P_{mn}(\omega)| \quad \text{when} \;\; (m,n) \equiv (0,2)  \pmod{3},
\end{equation*}
\begin{equation*}
|P_m(\omega)| |P_n(\omega)| = 4 \, |P_{mn}(\omega)| \quad \text{when} \;\; (m,n) \equiv (2,2) \pmod{3}.
\end{equation*}
\end{prop}

\begin{proof}
It is a consequence of the multiplicativity of sequence $r(n)/4$ observed above and of the following reformulation
of~\cite[Th.\,1.1]{KR1} or~\cite[Th.\,1.6\,(d)]{KR3}:
\begin{equation*}
|P_n(\omega)| = 
\begin{cases}
\hskip 12pt  r(n)  & \text{if} \; n \equiv 0,  \\
\hskip 7pt   r(n)/4  & \text{if} \; n \equiv 1, \quad\pmod{3}\\
\hskip 7pt r(n)/2   & \text{if} \; n \equiv 2.
\end{cases}
\end{equation*}
\end{proof}

We have the following consequence of Propositions\,\ref{prop-mult} and\,\ref{prop-mult1}.

\begin{coro}
Given a pair $(m,n)$ of coprime positive integers,
the polynomial $\Pg_{mn}(X)^2 - \Pg_m(X)^2 \Pg_n(X)^2$ is divisible by $X(X+1)(X^2-4)$,
and moreover by $X-1$ if $(m,n) \equiv (0,1), (1,1)$ or $(1,2) \pmod{3}$.
\end{coro}

We may wonder whether there are other integers $N$ such that $|\Pg_n(N)|$ is multiplicative. 
The answer is no, as witnessed by the following.

\begin{prop}
The function $|\Pg_n(N)|$ is multiplicative if and only if $N= -2, -1, 0$ or~$2$.
\end{prop}

\begin{proof}
The ``if direction'' is a reformulation of Proposition~\ref{prop-mult}. Conversely,
if $|\Pg_n(N)|$ is multiplicative, then we have $|\Pg_6(N)| = |\Pg_2(N)| |\Pg_3(N)|$ as a special case. 
Now the polynomial $\Pg_6(X)^2  -  \Pg_2(X)^2 \, \Pg_3(X)^2$ has the factorization
\begin{equation*}
\Pg_6(X)^2  -  \Pg_2(X)^2 \, \Pg_3(X)^2 = X(X-1)^3 (X+1)^3 (X-2)(X+2)^2 .
\end{equation*}
This implies that multiplicativity holds only for integers~$N$ such that $|N| \leq 2$.
We rule out the case $N = 1$ in view of Proposition~\ref{prop-mult1}.
\end{proof}

%%%
\section{The polynomials $F_n(X)$}\label{sec-aux}

We now define the monic Chebyshev polynomials mentioned above and the polynomials~$F_n(X)$. 

\subsection{Monic Chebyshev polynomials}\label{ssec-Tchb}

Given an integer $k\geq 0$ the standard Chebyshev polynomial of the first kind~$T_k(X)$ is the degree~$k$
polynomial defined by
\begin{equation*}
\cos(k \theta) = T_k(\cos (\theta)) ,
\end{equation*}
equivalently by
\begin{equation*}
\frac{q^k + q^{-k}}{2} = T_k \left( \frac{q + q^{-1}}{2} \right) .
\end{equation*}

We are interested in the following variant:
let $\TT_k(X)$ be the degree~$k$ polynomial defined by
\begin{equation}\label{defTT}
q^k + q^{-k} = \TT_k (q + q^{-1}).
\end{equation}
Both $T_k (X)$ and~$\TT_k(X)$ have integer coefficients.
They are related by
\begin{equation}\label{eq-TTT}
\TT_k (X) = 2 T_k(X/2) .
\end{equation}

The advantage of $\TT_k (X)$ over~$T_k(X)$ is that the former is monic for $k\geq 1$.
The \emph{monic Chebyshev polynomials}~$\TT_k (X)$ are denoted~$C_k$ in \cite{AS} 
and~$v_n(x)$ in~\cite{Ho}, where they are called \emph{Vieta--Lucas polynomials}.
See~\cite{Ri} for a general reference on Chebyshev polynomials.

Table\,\ref{tableTT} lists the polynomials~$\TT_k (X)$ for $1 \leq k \leq 12$.

\begin{table}[ht]
\caption{\emph{The monic Chebyshev polynomials $\TT_k (X)$}}\label{tableTT}
\renewcommand\arraystretch{1.25}
\noindent\[
\begin{array}{|c||c|}
\hline
k & \TT_k(X)  \\
\hline\hline
0 & 2  \\ 
\hline
1 & X  \\ 
\hline
2 & X^2 - 2   \\
\hline
3 & X^3 - 3X   \\ 
\hline
4 & X^4 - 4X^2 +2  \\ 
\hline
5 & X^5 - 5X^3 + 5X  \\ 
\hline
6 &  X^6 - 6X^4 + 9 X^2 - 2 \\
\hline
7 & X^7 - 7 X^5 + 14 X^3 - 7X  \\ 
\hline
8 & X^8 - 8 X^6 + 20 X^4 -16 X^2 +2 \\ 
\hline
9 &  X^9 - 9 X^7 + 27 X^5 - 30 X^3 + 9 X   \\
\hline
10 &  X^{10} - 10 X^8 + 35 X^6 - 50 X^4 + 25 X^2 - 2  \\
\hline
11 &  X^{11} - 11 X^9 + 44 X^7 - 77 X^5 + 55 X^3 - 11X   \\
\hline
12 &   X^{12} - 12 X^{10} + 54 X^8 - 112 X^6 + 105 X^4 - 36 X^2 +2  \\
\hline
\end{array}
\]
\end{table}

We now state a few properties of the polynomials~$\TT_k (X)$ which can be easily derived from well-known properties
of~$T_k (X)$. 
For instance, the polynomials~$\TT_k (X)$ can be defined inductively by
$\TT_0 (X) = 2$, $\TT_1 (X) = X$, and for $k\geq 1$ by
\begin{equation}\label{rec-TT}
\TT_{k+1}(X) = X\TT_k(X) - \TT_{k-1}(X) .
\end{equation}
The previous linear recurrence relation can be written in matrix form as
\begin{equation}\label{matrix-TT}
\begin{pmatrix}
\TT_{k+1}(X) \\ \TT_k(X)
\end{pmatrix}
=  M
\begin{pmatrix}
\TT_k(X) \\ \TT_{k-1}(X)
\end{pmatrix}, \;\; \text{where} \; M = 
\begin{pmatrix}
X & - 1\\
 1 & 0
\end{pmatrix} .
\end{equation}

The generating function of the polynomials~$\TT_k (X)$ is the power series expansion of the 
following rational function:
\begin{equation}\label{genfun-TT}
\sum_{k \geq 0} \, \TT_k (X) \, t^k = \frac{2 - Xt}{1 - Xt + t^2} .
\end{equation}

We have the following expression of~$\TT_n(X)$ in terms of powers of~$X$:
\begin{equation*}\label{eq-TX}
\TT_n(X) = \sum_{m=0}^{[n/2]} \, (-1)^m \left\{ \binom{n-m}{m} - \binom{n-m-1}{m-1} \right\} X^{n-2m} .
\end{equation*}

We also have a positivity result for the polynomials $\TT_k(X)$.

\begin{prop}\label{TTpos}
For each $k \geq 0$, we have $\TT_k (x) > 0$ if $x$ is a real number $\geq 2$.
\end{prop}

\begin{proof}
On one hand, as is well known,
the zeros of the standard Chebyshev polynomials~$T_k(X)$ all lie in the open interval $(-1,1)$.
Hence, in view of~\eqref{eq-TTT} the zeros of~$\TT_k(X)$ lie in $(-2,2)$. On the other, since $\TT_k(X)$ is monic,
$\lim_{x \to + \infty} \, \TT_k(x) = + \infty$. Combining these two facts implies the desired positivity.
\end{proof}

We close our summary on~$\TT_k (X)$ with the following result, which may be of independent interest.

\begin{prop}\label{prop-recTT}
For all $k\geq 0$ we have $\TT_k (X) = \tr(M^k)$, where $M$ is the $2 \times 2$-matrix introduced in \eqref{matrix-TT}.
\end{prop}

\begin{proof}
First observe that $\tr(M) = X = \TT_1(X)$ and $\tr(M^0) = \tr(\Id) = 2 = \TT_0(X)$, where $\Id$ is the identity matrix.
To prove the equality $\TT_k (X) = \tr(M^k)$, it suffices to show that 
the sequence $(\tr(M^k))_{k\geq 0}$ satisfies the linear recurrence relation~\eqref{rec-TT}. 
By the Cayley--Hamilton theorem applied to the matrix~$M$, we have
\[
M^2 = \tr(M) M - \det(M) \Id = XM - \Id .
\]
Now multiply both sides by~$M^{k-1}$ and take traces.
\end{proof}

\subsection{Definition and basic properties of $F_n(X)$}\label{ssec-F}

We now define the sequence of polynomials $(F_n(X))_{n\geq 0}$ inductively by $F_0(X) = 1$ and
\begin{equation}\label{def-F}
F_n(X) = F_{n-1}(X) + \TT_n(X)
\end{equation}
for $n\geq 1$. In other words,
\begin{equation}\label{def-F2}
F_n(X) = 1 + \sum_{k=1}^n \, \TT_k(X) .
\end{equation}
It is clear that each $F_n(X) $ has integer coefficients and is monic of degree~$n$.
We give a list of polynomials $F_n(X)$ for small~$n$ in Table~\ref{table-F}.

\begin{table}[ht]
\caption{\emph{The polynomials $F_n (X)$}}\label{table-F}
\renewcommand\arraystretch{1.25}
\noindent\[
\begin{array}{|c||c|}
\hline
n & F_n(X)  \\
\hline\hline
0 & 1  \\ 
\hline
1 & X + 1   \\
\hline
2 & X^2 + X - 1 \\ 
\hline
3 & X^3 + X^2 - 2X -1  \\ 
\hline
4 & X^4 + X^3 - 3 X^2 - 2X +1 \\ 
\hline
5 &  X^5 + X^4 - 4X^3 - 3 X^2 + 3X + 1  \\
\hline
6 & X^6 + X^5 - 5 X^4 -4 X^3 + 6X^2 + 3X -1 \\ 
\hline
7 & X^7 +X^6 - 6X^5 - 5 X^4 + 10 X^3 + 6 X^2 - 4X -1  \\ 
\hline
8 & X^8 + X^7 - 7 X^6 - 6 X^5 + 15 X^4 + 10 X^3 -10 X^2 -4 X +1  \\
\hline
&   X^9 + X^8 -8 X^7 -7 X^6  + 21 X^5 + 15 X^4  \\
9 &  - 20 X^3 - 10 X^2 + 5X + 1 \\
\hline
 &  X^{10} + X^9 - 9 X^8 - 8 X^7 + 28 X^6 + 21 X^5  \\
10 &   - 35 X^4 - 20 X^3 + 15 X^2 + 5X -1   \\
\hline
 & X^{11} + X^{10} - 10 X^9 - 9 X^8 + 36 X^7 + 28 X^6   \\
11 &  - 56 X^5 -35 X^4 + 35 X^3 + 15 X^2 - 6 X  - 1  \\
\hline
\end{array}
\]
\end{table}

Our first result is the following.

\begin{prop}\label{prop-gfF}
The generating function of the polynomials $F_n(X)$ is the power series expansion of the following rational function:
\[
\sum_{n \geq 0} \,  F_n(X) \,  t^n = \frac{1+t}{1-Xt +t^2} \, . 
\]
\end{prop}

\begin{proof}
Using~\eqref{genfun-TT}, we obtain
\begin{eqnarray*}
1 + \sum_{k \geq 1} \, \TT_k (X) \, t^k & = & \frac{2 - Xt}{1 - Xt + t^2} + 1 - \TT_0(X) \\
& = & \frac{2 - Xt}{1 - Xt + t^2} - 1 =  \frac{1 - t^2}{1 - Xt + t^2} .
\end{eqnarray*}
Therefore, 
\begin{eqnarray*}
\sum_{n \geq 0} \,  F_n(X) \,  t^n 
& = & \sum_{n \geq 0} \, \left( 1 + \sum_{m=1}^n \, \TT_m(X) \right) t^n \\
& = & \left( \sum_{\ell \geq 0} \, t^{\ell} \right) \left( 1 + \sum_{k \geq 1} \, \TT_k (X) \, t^k \right) \\
& = & \frac{1 - t^2}{(1-t) (1 - Xt + t^2)} =  \frac{1 + t}{1 - Xt + t^2} .
\end{eqnarray*}
\end{proof}

Setting $X= 0$ in the formula of Proposition~\ref{prop-gfF} yields the \emph{constant term} of~$F_n(X)$:
for all $m \geq 0$ we have
\begin{equation*}\label{constant-F}
F_{2m}(0) = F_{2m+1}(0) = (-1)^m .
\end{equation*}

For the remaining coefficients of~$F_n(X)$ we have the following expression.

\begin{prop}\label{prop-Fn}
For $n\geq 1$ we have
\begin{equation*}
F_n(X) = \sum_{0 \leq m \leq n/2} \, (-1)^m \binom{n - m}{m} X^{n-2m} \, +  
\hskip -10pt \sum_{0 \leq m \leq (n-1)/2} \, (-1)^m  \binom{n - m - 1}{m} X^{n-2m-1}  .
\end{equation*}
\end{prop}

\begin{proof}
Expanding the following fraction, we obtain
\begin{eqnarray*}
\frac{1}{1-Xt +t^2} & = & 1 + \sum_{r \geq 1} \, t^r (X - t)^r \\
& = & 1 + \sum_{r \geq 1} \sum_{s = 0}^r \, (-1)^s \binom{r}{s} X^{r-s} t^{r+s} \\
& = & 1 + \sum_{n \geq 1} \, \sum_{0 \leq m \leq n/2} \, (-1)^m \binom{n-m}{m} X^{n-2m} t^n .
\end{eqnarray*}
In view of Proposition~\ref{prop-gfF}, 
multiplying by $1+t$ yields the desired formula for~$F_n(X)$.
\end{proof}

\begin{coro}\label{coro-Fn}
For $n \geq 5$ we have 
\begin{eqnarray*}
F_n(X) & = & X^n + X^{n-1} - (n-1) X^{n-2} - (n-2) X^{n-3} + \\
&& {} + \frac{(n-2)(n-3)}{2} X^{n-4} + \frac{(n-3)(n-4)}{2} X^{n-5} + \\
&& {} + \text{monomials of degree} \leq n-6 .
\end{eqnarray*}
\end{coro}

In Section~\ref{pf-th-PFodd} we will need the following linear recurrence relation satisfied by the polynomials~$F_k(X)$.

\begin{prop}\label{prop-recF} We have 
\begin{equation*}
F_{k+1}(X) = X F_k(X) - F_{k-1}(X) \qquad (k\geq 1) .
\end{equation*}
\end{prop}

\begin{proof}
The rational function expressing the generating function in Proposition~\ref{prop-gfF} has the same denominator
as the rational function in~\eqref{genfun-TT}.
Therefore the polynomials $F_k(X)$ satisfy the same linear recurrence relation 
as the polynomials $\TT_k(X)$ in~\eqref{rec-TT} (see e.g. \cite[Chap.~6, Prop.~1.2 and Cor.~1.3]{BR}).
\end{proof}

%%%
\section{Main results for the polynomials $\Pg_n(X)$}\label{sec-resultsP}

We now connect  the polynomials $\Pg_n(X)$ with the polynomials~$F_k(X)$.
We start with an approximation result.

\subsection{Approximating $\Pg_n(X)$ with $F_{n-1}(X)$}\label{ssec-approx}

The aim of this subsection is the following result.

\begin{theorem}\label{th-approx}
For all $n\geq 2$, the difference $\Pg_n(X) - F_{n-1}(X)$ is a polynomial of degree~$< n/2 - 1$.
\end{theorem}

Recall that $\Pg_n(X)$ and $F_{n-1}(X)$ are both of degree~$n-1$. 
The theorem states that they coincide in degree $\geq n/2-1$.
The bound on degrees is sharp as can be seen from Table~\ref{table-PPTTF}:
indeed, $\Pg_9(X) - F_8(X) = - F_3(X) + F_1(X)$ is of degree~$3 < 9/2 - 1$; 
similarly, $\Pg_{11}(X) - F_{10}(X) = - F_4(X)$ is of degree~$4 < 11/2 - 1$,
and $\Pg_{15}(X) - F_{14}(X)$ is of degree~$6 < 15/2 - 1$.

Together with Corollary~\ref{coro-Fn} we derive the following.

\begin{coro}
For all $n \geq 10$ we have
\begin{eqnarray*}
\Pg_n(X) & = & X^{n-1} + X^{n-2} - (n-2) X^{n-3} - (n-3) X^{n-4} + \\
&& {} + \frac{(n-3)(n-4)}{2} X^{n-5} + \frac{(n-4)(n-5)}{2} X^{n-6} + \\
&& {} + \text{monomials of degree} \leq n-7 .
\end{eqnarray*}
\end{coro}

Recall the (nonnegative) coefficients $a_{n,i}$ of~$\Pg_n(X)$ appearing in \eqref{Pn-coeff}--\eqref{PPTT}.
For the proof of Theorem~\ref{th-approx} we need the following result.

\begin{lemma}\label{lem-ani}
We have $a_{n,i} = 1$ for all $i$ such that $n/2 - 1 \leq i \leq n-1$.
\end{lemma}

\begin{proof}
By~\cite[Thm.~1.3]{KR3} and \eqref{Pn-coeff} above
the integer $a_{n,i}$ is the number of divisors of~$n$ belonging to the half-open interval 
\[
I_{n,i} = \left( \frac{i + \sqrt{2n+ i^2}}{2}, i + \sqrt{2n+ i^2} \right] .
\]
It is enough to show that $n$ is the only divisor of~$n$ belonging to~$I_{n,i}$. Since any other divisor of~$n$ is $\leq n/2$,
it suffices to check that $n \leq i + \sqrt{2n+ i^2}$ and $n/2 \leq (i + \sqrt{2n+ i^2})/2$; the latter inequality  is equivalent to the former.
They are both equivalent to $n(n-2i) = (n - i)^2 - i^2 \leq 2n$, which simplifies into $n-2i \leq 2$ hence to $i \geq (n-2)/2= n/2 - 1$.
\end{proof}

\begin{proof}[Proof of Theorem~\ref{th-approx}]
By \eqref{def-F2} and~\eqref{PPTT} we have
\begin{eqnarray*}
\Pg_n(X) - F_{n-1}(X) 
& = & (a_{n,0} - 1)  + \sum_{1 \leq i \leq n - 1} \, (a_{n,i} -1) \, \TT_i(X) \\
& = & (a_{n,0} - 1)  + \sum_{1 \leq i < n/2 - 1} \, (a_{n,i} -1) \, \TT_i(X) ,
\end{eqnarray*}
the last equality being a consequence of Lemma~\ref{lem-ani}.
Since $\TT_i(X)$ is of degree~$i$, we obtain the desired result.
\end{proof}

\subsection{Odd divisors}\label{ssec-oddP}

We now give a complete formula for $\Pg_n(X)$ as a linear combination of the polynomials~$F_k(X)$, each appearance of~$F_k(X)$
being parameterized by an odd divisor of~$n$. 

Recall that if we express $n$ ($\geq 1$) as the product
$n = 2^a p_1^{a_1} p_2^{a_2} \cdots p_r^{a_r}$,
where $p_1, p_2, \ldots, p_r$ are distinct odd primes and $a, a_1, a_2, \ldots, a_r$ are nonnegative integers, 
then the number of \emph{odd divisors} of~$n$ (including~$1$) is equal to
\[
(a_1 + 1)(a_2 + 1) \cdots (a_r + 1) \geq 1.
\]
This number is a multiplicative function of~$n$ (see \cite[A001227]{OEIS}).

Observe that $(a_1 + 1)(a_2 + 1) \cdots (a_r + 1) = 1$ if and only if $n$ is a power of~$2$; 
otherwise, the number of odd divisors of~$n$ is at least~$2$.

For $n\geq 1$ and any odd divisor~$d$ of~$n$, define the integer 
\begin{equation}\label{eq-rnd}
r_n(d)  = \frac{n}{d} - \frac{d+1}{2} \, .
\end{equation}
When $d$ is the trivial divisor~$1$ we have $r_n(1) = n - 1$. 

We now state our main result for $\Pg_n(X)$.

\begin{theorem}\label{th-PFodd}
For all $n\geq 1$, 
\[
\Pg_n(X) = \sum_{d | n, \, d \, \mathrm{odd} \atop r_n(d) \geq 0} \, F_{r_n(d)}(X) 
- \sum_{d | n, \, d \, \mathrm{odd} \atop r_n(d) <0} \, F_{-r_n(d) -1}(X) .
\]
\end{theorem}

The number of terms~$F_r(X)$ in the previous formula is clearly equal to the number of odd divisors of~$n$.
The proof of Theorem~\ref{th-PFodd} will be given in Section~\ref{pf-th-PFodd}.

Since $d= 1$ is an odd divisor for any~$n\geq 1$ and $r_n(1) = n - 1 \geq 0$, the polynomial $F_{n-1}(X)$ always
appears with a positive sign in the previous formula for~$\Pg_n(X)$. 

We remarked above that only powers of~$2$ have a unique odd divisor, namely $d=1$.
This implies the following consequence of~Theorem~\ref{th-PFodd}, which explains the coincidences
$\Pg_n(N) = F_{n-1}(N)$  (set in boldface) observed in Table~\ref{table-values} for $N = 3$, $4$ and $5$.

\begin{coro}
We have $\Pg_n(X) = F_{n-1}(X)$ if and only if $n$ is a power of~$2$.
\end{coro}

As special cases of Theorem~\ref{th-PFodd} we obtain the following formulas for~$\Pg_n(X)$.

\subsubsection{Two odd divisors}

If $n = p$ is an odd prime, then it has two odd divisors, namely $1$ and $p$, so that
the formula for~$\Pg_p(X)$ has two summands. We have
\begin{equation*}
\Pg_p(X) = F_{p-1}(X) - F_{(p-3)/2}(X) = \sum_{m=(p-1)/2}^{p-1} \, \TT_m(X) .
\end{equation*}

Now consider the case $n = 2^a p$, where $a \geq 1$ and $p$ is an odd prime. 
Such an integer~$n$ has also only two odd divisors. 

If $r_n(p) = 2^a - (p+1)/2 \geq 0$, which is equivalent to $p \leq 2^{a+1} - 1$, then 
\begin{equation*}
\Pg_{2^a p}(X) = F_{2^a p-1}(X) + F_{2^a - (p+1)/2}(X)  .
\end{equation*}
For instance, $\Pg_{6}(X) = F_{5}(X) + F_{0}(X)$ and $\Pg_{12}(X) = F_{11}(X) + F_{2}(X)$

If $r_n(p) = 2^a - (p+1)/2 < 0$, which is equivalent to $p > 2^{a+1} - 1$, then 
\begin{equation*}
\Pg_{2^a p}(X) = F_{2^a p-1}(X) - F_{(p-1)/2 - 2^a}(X)  .
\end{equation*}
As special cases we have $\Pg_{10}(X) = F_{9}(X) - F_{0}(X)$ and $\Pg_{14}(X) = F_{13}(X) - F_{1}(X)$.

\subsubsection{Three odd divisors}

If $n = p^2$ for some odd prime~$p$, then
\begin{equation*}
\Pg_{p^2}(X) = F_{p^2-1}(X) - F_{(p^2 - 3)/2}(X) + F_{(p -1)/2}(X) .
\end{equation*}

\subsubsection{Four odd divisors}

If $n = p^3$ for some odd prime~$p$, then it has four odd divisors and we have
\begin{equation*}
\Pg_{p^3}(X) = F_{p^3-1}(X) - F_{(p^3 - 3)/2}(X) + F_{(2p^2 - p - 1)/2}(X) - F_{(p^2 -2p -1)/2}(X) .
\end{equation*}

Similarly, $n = 3p$ for some odd prime~$p \geq 7$ has four odd divisors, namely $1, 3, p, 3p$. Then
\begin{equation*}
\Pg_{3p}(X) = F_{3p-1}(X) - F_{3(p-1)/2} + F_{p-2}(X) - F_{(p-7)/2}  .
\end{equation*}
Belonging to this case are 
$\Pg_{15}(X) = F_{14}(X) - F_6(X) + F_3(X) + F_0(X)$
and $\Pg_{21}(X) = F_{20}(X) - F_9(X) + F_5(X) - F_0(X)$.

In Table~\ref{table-PPTTF} we have expressed each polynomial~$\Pg_n(X)$ ($1 \leq n \leq 16$) 
as a sum of the polynomials~$\TT_k(X)$ and as a linear combination of the polynomials~$F_k(X)$.

\begin{table}[ht]
\caption{\emph{Expressions for $\Pg_n (X)$ in terms of $\TT_k(X)$ and of $F_k(X)$}}\label{table-PPTTF}
\renewcommand\arraystretch{1.30}
\noindent\[
\begin{array}{|c||c|c|}
\hline
n &  \textrm{As sums of} \, \TT_k(X) & \textrm{In terms of}\, F_k(X) \\
\hline\hline
1 & 1 & F_0 \\ 
\hline
2 &  1 + \TT_1 & F_1 \\
\hline
3 &  \TT_1 + \TT_2 & F_2 - F_0 \\ 
\hline
4 &  1 + \TT_1 + \TT_2 + \TT_3 & F_3 \\ 
\hline
5 &  \TT_2 + \TT_3 + \TT_4 & F_4 - F_1 \\ 
\hline
6 &  \TT_0 + \TT_1 + \TT_2 + \TT_3 + \TT_4 + \TT_5  & F_5 + F_0 \\
\hline
7 &  \TT_3 + \TT_4 + \TT_5 + \TT_6 & F_6 - F_2 \\ 
\hline
8 &  1 + \TT_1 + \TT_2 + \TT_3 + \TT_4 + \TT_5 + \TT_6 + \TT_7  & F_7 \\ 
\hline
9 &  1 + \TT_1 + \TT_4 + \TT_5 + \TT_6 + \TT_7 + \TT_8 & F_8 - F_3 + F_1 \\
\hline
10 &  \TT_1 + \TT_2 + \TT_3 + \TT_4 + \TT_5 + \TT_6 + \TT_7  + \TT_8 + \TT_9 & F_9 - F_0 \\
\hline
11 &  \TT_5 + \TT_6 + \TT_7 + \TT_8 + \TT_9 + \TT_{10} & F_{10} - F_4  \\
\hline
&  \TT_0 + 2 \TT_1 + 2\TT_2 + \TT_3 + \TT_4 + \TT_5 & \\
12 &  + \TT_6 + \TT_7 + \TT_8 + \TT_9 + \TT_{10} +  \TT_{11} &  F_{11} + F_2 \\
\hline
13 &   \TT_6 + \TT_7 + \TT_8 + \TT_9 + \TT_{10} +  \TT_{11} +  \TT_{12} &  F_{12} - F_5 \\
\hline
&  \TT_2 + \TT_3 + \TT_4 + \TT_5 +  \TT_6 + \TT_7& \\
14 &  + \TT_8 + \TT_9 + \TT_{10} +  \TT_{11} +  \TT_{12} +  \TT_{13} &  F_{13} - F_1 \\
\hline
&  \TT_0 + \TT_1 + \TT_2 + \TT_3  + \TT_7 + \TT_8 & F_{14}(X) - F_6(X) \\
15 &   + \TT_9 + \TT_{10} +  \TT_{11} +  \TT_{12} +  \TT_{13} +  \TT_{14} &   {}+ F_3(X) + F_0(X)  \\
\hline
&  1 + \TT_1 + \TT_2 + \TT_3 + \TT_4 + \TT_5 +  \TT_6 + \TT_7 & \\
16 &  + \TT_8 + \TT_9 + \TT_{10} +  \TT_{11}  +  \TT_{12}  +  \TT_{13}  +  \TT_{14}  +  \TT_{15} &  F_{15}  \\
\hline
\end{array}
\]
\end{table}

\begin{rems}
(a) In Table~\ref{table-values} we set in italics the values $\Pg_n(N)$ and $F_{n-1}(N)$ which differ by~$1$, namely for $n = 3$, $6$
and $10$. These quasi-coincidences between values can be explained as follows.

We have
\[
\Pg_n(X) = F_{n-1}(X) + 1 = F_{n-1}(X) + F_0(X) 
\]
if and only if $n$ is an even \emph{perfect} number; equivalently,
$n = 2^a p$, where $p = 2^{a+1} - 1$ is a \emph{Mersenne prime} such as $3, 7, 31, 127$
(see \cite[A000668]{OEIS}). The lowest $n$ for which it holds are $n = 6, 28, 496, 8128$.

Similarly, 
\[
\Pg_n(X) = F_{n-1}(X)  - 1 = F_{n-1}(X) - F_0(X) 
\]
if and only if $n = 2^a p$, where $p = 2^{a+1} + 1$ is a \emph{Fermat prime} such as $3, 5, 17, 257$
(see \cite[A000215]{OEIS}). This works e. g. for $n = 3$, $10$, $136$, $32896$.

(b) In general, the formula of Theorem~\ref{th-PFodd} for $\Pg_n(X)$ contains $\pm F_0(X)$ as a summand if and only if 
$n$ is \emph{triangular}, i.e., $n = r(r+1)/2$ for some $r\geq 1$.
The lowest triangular integers are $1, 3, 6, 10, 15, 21, 28, 36, 45, 55, 66, 78, 91$.

(c) We may also wonder for which $n$ we have 
$\Pg_n(X) = F_{n-1}(X) \pm F_1(X)$.

We have $\Pg_n(X) = F_{n-1}(X) + F_1(X)$ if and only if $n = 2^a p$, where $p$ is a prime of the form $p = 2^{a+1} - 3$.
A list of such primes appear in \cite[A050415]{OEIS} and the corresponding~$n$ such as $n= 20, 104, 464, 1952$
in \cite[A181703]{OEIS}.

Similarly, $\Pg_n(X) = F_{n-1}(X) - F_1(X)$ if and only if $n = 2^a p$, where $p$ is a prime of the form $p = 2^{a+1} + 3$.
Primes $p = 5, 7, 11, 19$ are of this form (see~\cite[A057733]{OEIS}); they correspond to $n= 5, 14, 44, 152$.

More generally, $F_1(X)$ appears in the formula of Theorem~\ref{th-PFodd} for $\Pg_n(X)$ if and only if
$n = r(r+3)/2$ for some $r\geq 1$; see \cite[A000096]{OEIS} for a list of such~$n$.
\end{rems}

\begin{rem}
The sequence of integers $\Pg_n(3)$ is Sequence A329156 of~\cite{OEIS}.
The sequence $F_{n-1}(3)$ is A002878 in~\cite{OEIS}:
we have $F_{n-1}(3) = L(2n+1)$, where $L(n)$ is the Lucas sequence defined by $L(n) = L(n-1) + L(n-2)$, 
$L(0) = 2$ and $L(1) = 1$.

The sequence~$\Pg_n(4)$ is A386706.
According to~\cite[A001834]{OEIS}, the sequence $F_{n-1}(4)$ consist of the numerators of 
the lower principal convergents to~$\sqrt{3}$ in the sense of~\cite{Ki}.

The sequence~$\Pg_n(5)$ is A387017 and the sequence~$F_{n-1}(5)$ is A030221 in~\cite{OEIS}.
The latter consists of the values of Chebyshev even-indexed $U$-polynomials at~$\sqrt{7}/2$. 
\end{rem}

%%%
\section{Formulas for $C_n(q)$}\label{sec-Cnq}

In order to prove Theorem~\ref{th-PFodd} we need new expressions for the polynomial~$C_n(q)$.
Recall that $C_n(q)$ and $\Pg_n(X)$ are related by
\begin{equation*}
\frac{C_n(q)}{q^n} = \frac{(q-1)^2}{q} \, \frac{P_n(q)}{q^{n-1}} = \frac{(q-1)^2}{q} \, \Pg_n(q + q^{-1}) .
\end{equation*}
This section can be considered as an addendum to \cite{KR3},
where we computed the coefficients of~$C_n(q)$. 

\subsection{Increasing sequences}\label{ssec-IS}

The following material is taken from~\cite{Ma}.

We call \emph{increasing sequence} any finite nonempty ordered set of consecutive integers $\{ a+1, \ldots, a+h\}$ 
with smallest element~$a+1$ and largest element~$a+h$.
Here $a$ is any integer (positive, negative or zero) and $h$ is a positive integer. We denote this set by~$\IS(a,h)$;
it has $h$ elements. An increasing sequence~$\IS(a,h)$ is called \emph{positive}  if $a\geq 0$
and \emph{negative} if $a < 0$.
It is called \emph{odd} if $h$ is odd, and \emph{even} if $h$ is even.

We say that an integer~$n \geq 1$ is represented by~$\IS(a,h)$, and write $\IS(a,h) \models n$, if $n$ is equal to the sum
of all elements of~$\IS(a,h)$.

There is an involution $\IS(a,h) \mapsto \IS(\check{a},\check{h})$ on increasing sequences determined by
\begin{equation}\label{def-invIS}
a + \check{a} +1 = 0 \quad \text{and} \quad \check{a} + \check{h} = a + h.
\end{equation}
If $\IS(a,h)$ represents~$n$, so does~$\IS(\check{a},\check{h})$. 
Since $\check{h} - h = a - \check{a} = 2a + 1$, we see that $h$ and $\check{h}$ are of opposite parity.
In other words, $\IS(a,h)$ is even (resp. odd) if and only if $\IS(\check{a},\check{h})$ is odd (resp. even).
Moreover, if $\IS(a,h)$ is positive (resp. negative), then $\IS(\check{a},\check{h})$ is negative (resp. positive).

Let $n \geq 1$ be a positive integer.
To any \emph{odd} divisor~$d$ of~$n$ we associate the odd increasing sequence $\IS(a,h)$, where 
\begin{equation}\label{eq-ah}
a = \frac{n}{d} - \frac{d+1}{2}
\quad\text{and}\quad 
h = d;
\end{equation}
it is centered around~$n/d$:
\begin{equation}\label{def-ISd}
\IS(a,h) = \left\{ \frac{n}{d} - \frac{d-1}{2}, \ldots, \frac{n}{d} -1, \frac{n}{d}, \frac{n}{d} +1, \ldots, \frac{n}{d} + \frac{d-1}{2} \right\}.
\end{equation}
Depending on~$d$, this increasing sequence may be positive or negative. We have $\IS(a,h) \models n$.

From $a = n/d - (d+1)/2$ and $h = d$ we derive 
\begin{equation}\label{eq-ah2}
\check{a} = -\frac{n}{d} + \frac{d-1}{2}
\quad\text{and}\quad
\check{h} = \frac{2n}{d}.
\end{equation}
We thus obtain the even increasing sequence
\begin{equation}\label{def-ISdv}
\IS(\check{a},\check{h}) 
= \left\{ -\frac{n}{d} + \frac{d+1}{2}, \ldots, \frac{n}{d} + \frac{d-1}{2} \right\} \models n .
\end{equation}

Clearly, if $\IS(a,h)$ is positive, then $\IS(a,h) \subset \IS(\check{a},\check{h})$ and
\begin{equation}\label{def-ISdv1}
\IS(\check{a},\check{h}) \setminus \IS(a,h)
= \left\{ -\frac{n}{d} + \frac{d+1}{2}, \ldots, \frac{n}{d} - \frac{d+1}{2} \right\}  .
\end{equation}
If $\IS(a,h)$ is negative, then $\IS(\check{a},\check{h}) \subset \IS(a,h)$ and
\begin{equation}\label{def-ISdv2}
 \IS(a,h) \setminus \IS(\check{a},\check{h})
= \left\{ \frac{n}{d} - \frac{d-1}{2}, \ldots, -\frac{n}{d} + \frac{d-1}{2} \right\}  .
\end{equation}

For both pairs $(a,h)$ and $(\check{a}, \check{h})$ we have the same relation, namely
\begin{equation}\label{eq-IS}
\frac{h(h+2a+1)}{2} = n
= \frac{\check{h}(\check{h}+2\check{a}+1)}{2}  .
\end{equation}

Conversely, starting from an increasing sequence $\IS(a,h)$ we obtain an odd divisor $d$ of~$n$ as follows: 
$d$ is equal to the only element of~$\{h, \check{h}\}$ which is odd.
In this way we have a bijection between the set of odd divisors of~$n$ and 
the set of positive increasing sequences representing~$n$,
which is in bijection with the set of negative increasing sequences representing~$n$.

Suppose that given $n\geq 1$ we have a pair $(a,h)$ of integers with $h\geq 1$ verifying Equation~\eqref{eq-IS},
which we rewrite as $h(h+2a+1) = 2n$.
If $h$ is odd, hence coprime to~$2$, then $d = h$ is an odd divisor of~$n$ and $a = n/d - (d+1)/2$, as in~\eqref{eq-ah}.
If $h$ is even, then $d = h+2a+1$ is odd; being a divisor of~$2n$, it is an odd divisor of~$n$; in this case
$h = 2n/d$ and $a = - n/d + (d-1)/2$, as in~\eqref{eq-ah2}.

\begin{exas}
(i) If $n$ is a power of~$2$, then it has only one odd divisor, namely $d=1$; the corresponding positive increasing sequence
is $\IS(n-1,1) = \{n\}$ and the corresponding negative one is 
\[
\IS(-n,2n) = \{ -(n-1), \ldots, -1,0,1, \ldots, n-1,n\}.
\]

(ii) Any other integer~$n$ has odd divisors $d\neq 1$, hence more than one positive (or negative) increasing sequences.

Take for instance $n= 18$. If $d = 3$, then $a = 4$, $h=3$, $\check{a} = -5$, $\check{h} = 12$ and 
\[
\IS(4,3) = \{ 5,6,7\} \quad \text{and}\quad  
\IS(-5, 12) = \{-4, -3, \ldots, 3, 4, 5, 6, 7\} .
\]

If $d = 9$, then $a = -3$, $h=9$, $\check{a} = 2$, $\check{h} = 4$ and 
\[
\IS(-3,9) = \{ -2, -1, 0, 1, 2, 3, 4, 5,6 \} \quad \text{and}\quad  
\IS(2,4) = \{ 3, 4, 5, 6\} .
\]
\end{exas}

\subsection{A formula for~$C_n(q)$ in terms of odd divisors of~$n$}\label{ssec-Cnq}

In \cite[Th.~1.1]{KR3} we gave the following explicit expression for~$C_n(q)$:
\begin{equation*}\label{Cn-coeffi}
\frac{C_n(q)}{q^n} = c_{n,0}  + \sum_{i=1}^n \, c_{n,i} \, \left( q^{i} + q^{-i} \right) , 
\end{equation*}
where
\begin{equation}\label{formula-cn0}
c_{n,0} = 2 (-1)^r
\end{equation}
if $n = r(r+1)/2$ for some positive integer~$r$ and $c_{n,0} = 0$ otherwise. 
For the remaining coefficients $c_{n,i}$ ($i\geq 1$) we have
\begin{equation}\label{formula-cni}
c_{n,i} =
\left\{
\begin{array}{cl}
(-1)^k & \text{if}\;\, n = k(k+2i +1)/2 \;\; \text{for some integer}\; k  \geq 1, \\
\noalign{\smallskip}
(-1)^{k-1} & \text{if}\;\, n = k(k+2i -1)/2 \;\; \text{for some integer}\; k \geq 1, \\
\noalign{\smallskip}
0 & \text{otherwise.} 
\end{array}
\right.
\end{equation}

Observe that in (\ref{formula-cni}) the first two conditions are mutually exclusive, 
as implied by~\cite[Lemma~3.5 (ii)]{KR3}. This may also be verified as follows: suppose
\[
k(k+2i+1)=h(h+2i-1)
\] 
with $i,h,k\geq 1$; then $k^2+2ik+k=h^2+2ih-h$, hence 
\[
(k-h+1)(k+h)+2i(k-h)=0 .
\] 
It is readily seen that the left-hand side cannot vanish, neither in the case $k\geq h$, nor in the case~$k<h$.

\begin{lemma}\label{formula1} 
We have
\[
\frac{C_n(q)}{q^n} = \sum\,  (-1)^h \, \left( q^a + q^{-a} \right), 
\]
where the summation is over all $h\in\ZZ_{\geq 1}$ and $a\in \ZZ$ such that $n=h(h+1+2a)/2$. 
Moreover, if $n$ has such a representation and $|a|$ is given, then $a,h$ are unique.
\end{lemma}
 
\begin{proof}
From~\eqref{formula-cn0} it follows that both sides have the same constant term. 

From now on we concentrate on the case when $a$ is a nonzero integer.
Let $i\geq 1$.
Suppose that $n=k(k+2i +1)/2$ with $k\geq 1$ as in the first row of~\eqref{formula-cni}; 
then $n=h(h+2a +1)/2$ with $h=k \geq 1$, $a=i\in \ZZ \setminus \{0\}$, and $(-1)^k=(-1)^h$. 

Now suppose that $n = k(k+2i -1)/2$ with $k\geq 1$ as in the second row of~\eqref{formula-cni}; 
then $n=h(h+2a +1)/2$ with $h=k+2i-1\geq 1$, $a=-i\in \ZZ \setminus \{0\}$, and $(-1)^{k-1}=(-1)^h$. 

Conversely, let $n=h(h+2a +1)/2$ with $h\geq 1$, $i\in \ZZ$ and $a\neq 0$. 
Then either $a\geq 1$, and then $n=k(k+2i +1)/2$ with $k=h$, $i=a\geq 1$, and $(-1)^k=(-1)^h$; 
or $a < 0$, and then $n=k(k+2i-1)/2$, $k=h+2a+1\geq 1$, $i = -a > 0$, and $(-1)^{k-1}=(-1)^h$.

This proves the lemma in view of the observation made just before.
\end{proof}

We have the following new formulas for~$C_n(q)$.

\begin{theorem}\label{th-Cn}
(i) For each integer $n\geq 1$ we have 
\[
\frac{C_n(q)}{q^n} = \sum \, (-1)^h \, \left( q^a + q^{-a} \right),
\] 
where the summation is taken over all increasing sequences $\IS(a,h)$ such that $\IS(a,h) \models n$. 
Moreover, for given $|a|$, there is a unique increasing sequence $\IS(a,h)$ with $\IS(a,h) \models n$. 

(ii) We also have
\[
\frac{C_n(q)}{q^n} = \sum_{d | n , \, d \, \mathrm{odd}} \, 
\left( q^{\frac{n}{d} - \frac{d-1}{2}} +  q^{-\frac{n}{d} + \frac{d-1}{2}}  
- q^{\frac{n}{d} - \frac{d+1}{2}}  - q^{-\frac{n}{d} + \frac{d+1}{2}} \right) .
\] 
Moreover, all monomials $\neq q^0$ in the previous sum are distinct.
\end{theorem}

\begin{proof} 
(i) The first formula follows from Lemma~\ref{formula1} since the sum of the elements in $\IS(a,h)$ is 
equal to 
\[
a+1+a+2+ \cdots+a+h = ha+h(h+1)/2 = h(h+2a+1)/2.
\]

(ii) As in \S~\ref{ssec-IS}, with each odd divisor~$d$ of~$n$ we associate the two increasing 
sequences $\IS(a,h)$ and $\IS(\check a,\check h)$, respectively odd and even, defined by \eqref{def-ISd} and \eqref{def-ISdv}. 
All increasing sequences representing~$n$ are obtained in this way. 
Item~(ii) then follows from Item~(i): the first two terms correspond to~$\IS(\check{a}, \check{h})$ and the last two to~$\IS(a,h)$.
\end{proof}

The previous result immediately implies the following.

\begin{coro}\label{coro-vacni} 
The sum of the absolute values of the coefficients of $C_n(q)$ is equal to four times the number of odd divisors of~$n$.
\end{coro}

\subsection{Zeta functions}\label{ssec-zeta}

We state further consequences of Theorem~\ref{th-Cn}. 
Recall the \emph{local zeta function}~$Z_{\HH^n/\FF_q}(t)$ 
of the Hilbert scheme~$\HH^n$ of $n$~points on the two-dimensional torus
over~$\FF_q$ mentioned in the introduction. 
In \cite[Th.~2.1]{KR3}, we gave the following formula for it: 
\[
Z_{\HH^n/\FF_q}(t) = \frac{1}{(1-q^nt)^{c_{n,0}}} \, \prod_{i=1}^n \, \frac{1}{[(1-q^{n+i}t)(1-q^{n-i}t)]^{c_{n,i}}} \, ,
\]
where $c_{n,i}$ are the coefficients of~$C_n(q)$.
Rewriting this formula in view of Theorem~\ref{th-Cn} yields the following statement, where $r_n(d) = n/d - (d+1)/2$
is the integer introduced in~\eqref{eq-rnd}.

\begin{prop}\label{prop-zeta}
For $n\geq 1$ we have
\begin{equation*}
Z_{\HH^n/\FF_q}(t) 
= \prod_{d | n , \, d \, \mathrm{odd}} \, 
\frac{(1-q^{n + r_n(d)}t)(1-q^{n - r_n(d)}t)}{(1-q^{n + r_n(d)+1}t)(1-q^{n - r_n(d)-1}t)} .
\end{equation*}
\end{prop}

This is a rational function whose numerator and denominator are factorized into degree-one polynomials.
Clearly the number of such factors in the numerator as well as in the denominator is equal to twice
the number of odd divisors of~$n$.

Let us display the product above in two examples. 
For $n=4$, which has only~$1$ as an odd divisor, we have $r_4(1) = 3$. Hence
\begin{equation*}
Z_{\HH^4/\FF_q}(t)  = \frac{(1-q^{7}t)(1-qt)}{(1-q^{8}t)(1-t)} .
\end{equation*}
For $n=3$, which has two odd divisors, we have $r_3(1) = 2$ and $r_3(3) = -1$. Hence
\begin{equation*}
Z_{\HH^3/\FF_q}(t)  = \frac{(1-q^{5}t)(1-qt)}{(1-q^{6}t)(1-t)} \cdot \frac{(1-q^{2}t)(1-q^4t)}{(1-q^{3}t)^2} .
\end{equation*}

Proposition~\ref{prop-zeta} allows us to give a similar formula for the \emph{Hasse--Weil zeta function} 
$\zeta_{\HH^n}(s)$ of the above Hilbert scheme; this zeta function is defined over the complex numbers by
\begin{equation*}\label{def-HWZ}
\zeta_{\HH^n}(s) = \prod_{p \; \text{prime}} \, Z_{{\HH^n_{\FF_p}}/\FF_p}(p^{-s}), \qquad (s \in \CC)
\end{equation*}
where the product is taken over all prime integers~$p$. 
As an immediate consequence of the formula for the local zeta functions~$Z_{\HH^n/\FF_q}(t)$ we have
\begin{equation}
\zeta_{\HH^n}(s) = \prod_{d | n , \, d \, \mathrm{odd}} \, 
\frac{ \zeta(s - n - r_n(d)) \, \zeta(s - n + r_n(d) ) }{ \zeta(s - n - r_n(d) -1) \, \zeta(s - n + r_n(d) +1) } \, , 
\end{equation}
where $\zeta(s)$ is Riemann's zeta function.
Recall from \cite[Sect.~2]{KR3} that the Hasse--Weil zeta function satisfies the simple functional equation
\begin{equation}
\zeta_{\HH^n}(s) = \zeta_{\HH^n}(2n-s).
\end{equation}

%%%
\section{Proof of Theorem~\ref{th-PFodd}}\label{pf-th-PFodd}

The following lemma exhibits a crucial link between the polynomials~$F_r(X)$ and
the expression in parentheses in the formula of Theorem~\ref{th-Cn}\,(ii).

\begin{lemma}\label{lem-CF}
For $r \geq 0$ we have
\[
q^{r+1} + q^{-r-1} - q^{r} - q^{-r}  = q^{-1}(q-1)^2 F_r(q + q^{-1}) .
\]
\end{lemma}

\begin{proof}
The formula can be rephrased in terms of polynomials in the indeterminate~$X$
as $\TT_{r+1}(X) - \TT_r(X) = (X-2) F_r(X)$. 
By \eqref{def-F} the latter is equivalent to $(F_{r+1}(X) - F_r(X)) - (F_r(X) - F_{r-1}(X)) = X F_r(X) - 2 F_r(X)$,
hence to the linear recurrence relation of Proposition~\ref{prop-recF}.
\end{proof}

\begin{proof}[Proof of Theorem~\ref{th-PFodd}]
It follows from Theorem~\ref{th-Cn}~(ii) that
\begin{eqnarray*}
q^{-1}(q-1)^2 \frac{P_n(q)}{q^{n-1}}  
& = & \frac{C_n(q)}{q^n} \\
& = & \sum_{d | n , \, d\, \mathrm{odd}} \, 
\left( q^{\frac{n}{d} - \frac{d-1}{2}} +  q^{-\frac{n}{d} + \frac{d-1}{2}}  
- q^{\frac{n}{d} - \frac{d+1}{2}}  - q^{-\frac{n}{d} + \frac{d+1}{2}} \right) .
\end{eqnarray*}
We partition the previous sum into two sums depending on the sign of $r_n(d) = n/d - (d+1)/2$.
We obtain 
\begin{eqnarray*}
q^{-1}(q-1)^2 \frac{P_n(q)}{q^{n-1}}  
& = & \sum_{d | n, \, d \, \mathrm{odd} \atop r_n(d) \geq 0} \,
\left( q^{\frac{n}{d} - \frac{d-1}{2}} +  q^{-\frac{n}{d} + \frac{d-1}{2}}  
- q^{\frac{n}{d} - \frac{d+1}{2}}  - q^{-\frac{n}{d} + \frac{d+1}{2}} \right) \\
& & + \sum_{d | n, \, d \, \mathrm{odd} \atop r_n(d) < 0} \,
\left( q^{\frac{n}{d} - \frac{d-1}{2}} +  q^{-\frac{n}{d} + \frac{d-1}{2}}  
- q^{\frac{n}{d} - \frac{d+1}{2}}  - q^{-\frac{n}{d} + \frac{d+1}{2}} \right) .
\end{eqnarray*} 
Now replacing $r$ by $r_n(d)$ in Lemma~\ref{lem-CF}, we see that the sum of monomials corresponding to $r_n(d) \geq 0$
is equal to 
\[
q^{-1}(q-1)^2 \sum_{d | n, \, d \, \mathrm{odd} \atop r_n(d) \geq 0} \, F_{r_n(d)}(q + q^{-1}).
\]
Replacing $r$ by $-r_n(d) - 1$ in Lemma~\ref{lem-CF}, we observe that the sum of monomials corresponding to $r_n(d) < 0$, 
equivalently to $-r_n(d) - 1 \geq 0$, is equal to
\[
q^{-1}(q-1)^2 \sum_{d | n, \, d \, \mathrm{odd} \atop r_n(d) < 0} \, F_{-r_n(d) -1}(q + q^{-1}).
\]
To conclude, it suffices to divide by $q^{-1}(q-1)^2$ 
and to remind that $P_n(q)/q^{n-1} = \Pg_n(q + q^{-1})$.
\end{proof}

We close this section with another formula for $P_n(q)/q^{n-1}$.

\begin{prop}\label{formula2}
We have
\[
\frac{P_n(q)}{q^{n-1}} =
\sum_{d|n, \,d\,\mathrm{odd} \atop r_n(d) \geq 0} \; \sum_{i \in  \IS(\check{a},\check{h}) \setminus \IS(a,h)} \, q^i
- \sum_{d|n, \,d\,\mathrm{odd} \atop r_n(d) < 0}  \; \sum_{i \in  \IS(a,h)\setminus   \IS(\check{a},\check{h})}\, q^i  ,
\]
where $a = r_n(d)$ and $\check{a} = - r_n(d) - 1$, and
$\IS(a,h)$ and $\IS(\check{a},\check{h})$ are defined respectively by \eqref{def-ISd} and \eqref{def-ISdv}.
\end{prop}

\begin{proof}
It follows from Theorem~\ref{th-PFodd} and the explicit descriptions~\eqref{def-ISdv1}--\eqref{def-ISdv2}.
\end{proof}

\begin{rem}
Sums of powers of $q$ indexed by odd divisors have also appeared in work of J. M. R. Caballero,
especially in Chapter~II of~\cite{RC1}.
\end{rem}

%%%
\section*{Acknowledgement}
We are grateful to Jos\'e Bastidas Olaya for his help with SageMath
and to V\'aclav Kot\v{e}\v{s}ovec for having corrected the value~$\Pg_7(4)$ in Table~\ref{table-values}.
We thank Christophe Pouzat and Marcus Slupinski for contributing to the present title.

%%%%%%%%%%%%

\end{document}